\documentclass[11pt]{article}
\usepackage{mathrsfs}
\usepackage{amsfonts}
\usepackage{amsmath}
\usepackage{amssymb}
\usepackage{amsthm}
\usepackage{color}
\usepackage{mathrsfs, amssymb}
\usepackage{subfigure}
\usepackage{float}
\usepackage{comment}
\usepackage{graphicx}
\usepackage{psfrag}
\usepackage{epstopdf}
\usepackage[pagewise]{lineno} 

\allowdisplaybreaks[4]

\newcommand{\dy}{\mathrm{d} y}
\newcommand{\dx}{\mathrm{d} x}

\renewcommand{\d}{\mathrm{d}}
\newcommand{\ee}{\mathrm{e}}

\newcommand{\R}{\mathbb{R}}

\numberwithin{equation}{section}
\newtheorem{theorem}{Theorem}[section]

\newtheorem{lemma}[theorem]{Lemma}

\newtheorem{proposition}[theorem]{Proposition}

\begin{document}

\title {\bf Persistence of solutions in a nonlocal predator-prey system with a shifting habitat}
\author{Min Zhao
\thanks{Corresponding author. Email addresses: minzhao@mail.bnu.edu.cn (M. Zhao), ryuan@bnu.edu.cn (Rong Yuan)}
, Rong Yuan}
\date{\it Laboratory of Mathematics and Complex Systems (Ministry of Education), \\School of Mathematical Sciences, Beijing Normal University, \\Beijing 100875, People's Republic of China}
\maketitle

\vspace*{-1cm}
\begin{center}

\begin{minipage}[t]{13cm}

{\bf Abstract:}\quad  In this paper, we mainly study the propagation properties of a nonlocal dispersal predator-prey system in a shifting
environment. It is known that Choi et al. [J. Differ. Equ. 302 (2021), pp. 807--853] studied the persistence or extinction of the prey and the predator separately in various moving frames. In particular, they achieved a complete picture in the local diffusion case. However, the question of the persistence of the prey and the predator in some intermediate moving frames in the nonlocal diffusion case is left open in Choi et al.'s paper. By using some prior estimates, the Arzel\`{a}-Ascoli theorem and a diagonal extraction process, we can extend and improve the main results of Choi et al. to achieve a complete picture in the nonlocal diffusion case.

{\bf Keywords}\quad Predator-prey system; Persistence; Nonlocal dispersal; Shifting environment\\
{\bf Mathematics Subject Classification}\quad 35K57, 35K55, 35B40, 92D25
\end{minipage}
\end{center}

\section{Introduction}

In recent years, in addition to seasonal and regional differences, climate changes caused by global warming, industrialization and overdevelopment have had a huge impact on the habitats of biological species \cite{Wu19}. The habitat of species will move in time and space with climate change \cite{Wu13, Wu19}. A simple pattern for measuring climate change is the shifting of environment quality with constant speed.
This will translate into the shifting of the habitat quality, which will be reflected by the shifting of the growth rate for a species \cite{Yuan19}. At present, many scholars have devoted themselves to studying such topics, see \cite{Berestycki09, Berestycki18, Du15, Gonzalez10, Hu15, Hu17, Lei17, Li14, Li16, Potapov04, Parr12, Scheffer12, Vo15, Wang-Li21, Zhou11}.

Random diffusion describes that organisms can only move to their surrounding neighborhoods \cite{Bao16, Kao}. Based on this property of random diffusion, we can describe the dynamics of random dispersal through the reaction-diffusion
model. There has been much research on the spreading speeds and forced waves of the reaction-diffusion systems in a shifting environment. For the case of  scalar equation, Bouhours et al. \cite{Bouhours19}, Hu et al. \cite{Hu19} and Li et al. \cite{Li14} explored conditions for extinction and persistence as well as spatial-temporal dynamics of the species with a shifting habitat edge, respectively.
Fang et al. \cite{Fang21} and  Vo \cite{Vo15} investigated the propagation dynamics of a reaction-diffusion equation in a time-periodic shifting environment. Berestycki et al. \cite{Berestycki09, Berestycki08, Berestycki18} established existence of forced waves of reaction-diffusion equations for population dynamics with a shifting habitat. Hu et al. \cite{Hu17} established existence of an extinction wave in the Fisher equation with a shifting habitat. For the case of competition model, Potapov and Lewis \cite{Potapov04} considered a Lotka-Volterra competition model in a domain with a moving range boundary, by which they obtained a critical patch size for each species to persist and spread. Yuan et al. \cite{Yuan19} and Zhang et al. \cite{Zhang17} studied the persistence versus extinction for two competing species under climate change. Dong et al. \cite{Dong21} investigated the existence of forced waves in a Lotka-Volterra competition-diffusion model with a shifting habitat. Berestycki et al. \cite{Berestycki14} investigated the persistence of two species and the gap formation of Lotka-Volterra competition model. For the case of cooperative model, Yang et al. \cite{Yang19} considered the existence and asymptotics of forced wave solutions in a Lotka-Volterra cooperative model under climate change. Furthermore, we refer the readers to some literature considering forced waves in domains with a free boundary \cite{Hu20,Lei17}.

In addition to the random diffusion, nonlocal dispersal is more reasonable for some species to travel for some distance, and their movement and interactions may occur between non-adjacent
spatial locations \cite{Bao16,Jin, Kao, Lee, Levin,Zhao22}. The widespread long-distance dispersal or nonlocal internal interactions are usually modeled by an appropriate integral operator, such as $\int_{\R}J(x-y)[u(y)-u(x)]\dy$. 
Concerning the study of the effects of climate change in nonlocal dispersal, the work of Coville \cite{Coville}, Leenheer et al. \cite{Leenheer20}, Li et al. \cite{Li18}, and Wang et al. \cite{Wang19} studied the persistence criterion and existence, uniqueness as well as stability of forced waves for the scalar nonlocal dispersal population model in a shifting environment. In addition, Zhang et al. \cite{Zhang20} explored the propagation dynamics of a nonlocal dispersal Fisher-KPP equation in a time-periodic shifting habitat. For the case of competition system, Wu et al. \cite{Wu19} studied the spatial-temporal spreading dynamics of a Lotka-Volterra competition model with nonlocal dispersal under a shifting environment. Wang et al. \cite{Wang20} and Wang et al. \cite{Wang21} investigated the existence of forced waves and gap formations for the lattice and continuous Lotka-Volterra competition models with nonlocal dispersal and shifting habitats, respectively. Bao et al. \cite{Bao16} studied the traveling wave solutions of Lotka-Volterra competition systems with nonlocal dispersal in
periodic habitats. 
For the case of the prey-predator system, Choi et al. \cite{Choi21} studied the persistence of species in a predator-prey system with climate change and either nonlocal or local dispersal.

Due to the lack of comparison principle and the issue of the compactness of the set of solutions with bounded initial data, there is less work on the predator-prey systems with nonlocal dispersal in shifting environments. Motivated by the aforementioned works, we would like to extend and improve the work of Choi et al. \cite{Choi21} to deal with the spreading population dynamics of predator-prey species in some intermediate moving frames with nonlocal dispersal under a shifting habitat. In this paper, we consider the following predator-prey model with nonlocal dispersal proposed by Choi et al. \cite{Choi21}
\begin{equation}\label{1}
\left\{\begin{array}{l}
{\frac{\partial u}{\partial t}(x, t)=d_{1}[(J_{1} * u)(x, t)-u(x, t)]+ r_{1}u(x, t)[\alpha(x-st)-u(x, t)-av(x, t)],} \\
{\frac{\partial v}{\partial t}(x, t)=d_{2}[(J_{2} * v)(x, t)-v(x, t)]+ r_2v(x, t)[-1+b u(x, t)-v(x, t)],}
\end{array}\right.
\end{equation}
where $x\in \mathbb{R}, t>0$ and $r_{1}, r_{2}, a, b$ are all positive constants. Here $u(x, t)$ and $v (x, t)$ are the population
densities of prey and predator species at spatial position $x \in \mathbb{R}$ and time $t>0$, respectively. The dynamics of the prey
population follow a logistic growth which depends on the shifting habitat with a fixed speed $s>0$. Parameters $r_{1}$ and $r_2$ denote the intrinsic growth rates. The constant $a$ denotes the predation rate and $b$ denotes the biomass conversion rate. $d_{1}>0$ and $d_{2}>0$ are the diffusion coefficients for prey and predator species, respectively. The term $J_{i}\ast w-w$
describes the spatial dispersal process and
$$
(J_{i}\ast w)(x, t)-w(x, t)=\int_{\mathbb{R}} J_{i}(x-y) w(y, t) \mathrm{d} y-w(x, t),\quad i=1, 2,
$$
where the symbol $*$ denotes the convolution product for the spatial variable. Here we assume that the kernel function $J_{i}: \R\rightarrow\R (i=1, 2)$ is continuous and satisfies the following properties:
\begin{description}
  \item[(J1)] $\displaystyle J_i(x) = J_i(-x)\geq0 \text{ for any }x\in \R \text{ and } \int_{\R} J_i(x)\dx =1, i=1, 2$;
  \item[(J2)] $\displaystyle J_i\in C^1(\R)$ \text{ and } $J_i$ is compactly supported, i=1, 2.
\end{description}
The function $\alpha(\cdot)$ models the climate change which depends on a shifting variable, and throughout the paper we assume that it satisfies the following properties:
\begin{description}
  \item[($\alpha_1$)] $\alpha(\cdot)\in C^1(\R)$ and nondecreasing in $\R$;
  \item[($\alpha_2$)] $-\infty<\alpha(-\infty)< 0 < \alpha(\infty) < \infty$; furthermore, we choose $\alpha(\infty) = 1$ without loss of
   generality (up to a rescaling);
  \item[($\alpha_3$)] The derivative of $\alpha(\cdot)$ is bounded in $\R$.
\end{description}

It is clearly that the shifting environment may be divided into the favorable region $\{x \in\R :\alpha (x -st) > 0\}$ and the unfavorable region $\{x \in\R :\alpha (x -st) \leq 0\}$, both shifting with speed $s>0$. The non-decreasing property of $\alpha(\cdot)$ assumes that the environment gets worse as time goes on, and the negativity of $\alpha(-\infty)$ accounts for a scenario in that the environment is shifting to a very severe level in the unfavorable region. In assumption $(\alpha_3)$, we only require the derivative of $\alpha(\cdot)$ to be bounded instead of uniformly continuous, mainly considering that more dramatic effects of global warming, such as changes in the frequencies of severe rainstorms, hurricanes
and other climatic disasters, may have a short-term consequence on local or regional species survival and reproduction \cite{Johnsgard, Lewis}.

System \eqref{1} is supplemented by the initial conditions:
\begin{equation}\label{2}
u(x, 0)=u_{0}(x), \quad v(x, 0)=v_{0}(x),\quad x \in \mathbb{R},
\end{equation}
where $u_{0}(x), v_{0}(x)$ are bounded nonnegative functions with nonempty compact support. 

Throughout this work, we give the following assumption about the parameter:
\begin{description}
  \item[(H1)] $b>1$,
\end{description}
which ensures that the amount of prey is sufficient to maintain the positive density of the predator. From system \eqref{1}, we can see that the predator cannot survive without the prey.

It is known that Choi et al. \cite{Choi21} explored the propagation properties of the predator and the prey of system \eqref{1} in two different
situations, namely, the prey is faster than the predator, in the sense that its maximal speed $s^*$ is larger than the maximal speed $s_*$ of the predator and the situation $s^*\leq s_*$, i.e., the predator is faster than the prey, respectively. Here, the speed $s^*$ represents the maximum speed of the prey in the ``favorable'' environment and without predator; the speed $s_*$ denotes the maximum speed of the predator when the prey density is at saturation, see Section 2 for details. For the persistence of system \eqref{1}, they showed that the predator and the prey are persistent in cases $s<s^{**}$ $(s^{**}<s^{*})$ and $s<\underline{s}^{*}=\min\{s^{**}, s_{**}\}$ $(s_{**}<s_{*})$, respectively. Here $s^{**}$ is the speed of the prey in the favorable environment when there is a maximal amount of predator, and $s_{**}$ is the speed of the predator when there is a minimal amount of prey. However, the proof of persistence of the predator and the prey of system \eqref{1} with speed $s$ and $\min\{s_*, s^*\}>s$ remains an open question. In this paper, we extend and improve the main results in Choi et al. \cite{Choi21} to show that both species always persist and achieve a complete picture of the spreading dynamics of \eqref{1}. Inspired by the work of Choi et al.~\cite{Choi21} and Zhang et al.~\cite{Zhang}, we use some prior estimates, the Arzel\`{a}-Ascoli theorem and a diagonal extraction process to perform various limiting arguments. We conclude that under certain conditions, the predator and the prey persist.

This paper is organized as follows. In the next section, we establish some preliminary results.
In Section 3, we mainly consider the persistence of the prey $u$ of the system \eqref{1} with initial value \eqref{2} in the moving frames with speeds between $s$ and $\underline{s}= \min\{s^*, s_*\} > s.$ In Section 4, we show the persistence of the predator $v$ in the moving frames with speeds between $s$ and $\underline{s}= \min\{s^*, s_*\} > s.$

\section{Preliminaries }

In this section, we mainly introduce some preliminaries and recall some results for spreading speeds.

Firstly, we define 
$$X=\{w(x)|~w(x): \mathbb{R}\rightarrow\mathbb{R}\text{ is bounded and uniformly continuous}\},$$
with the norm
$$\|w\|_{X}=\displaystyle\sup_{x\in \mathbb{R}}|w(x)|.$$
Then $(X, \|\cdot\|_{X})$ is a Banach space.
Furthermore, for any constant $d > 0$, let
$$X_{d}=\{w \in X: 0 \leq w(x) \leq d,\, \forall x \in \mathbb{R}\}.$$
Set the order of the space $X^{2}=X\times X$ as follows:
$$\underline{w}\leq \overline{w}\Leftrightarrow \underline{w}_{i}(x)\leq \overline{w}_{i}(x), \quad x\in \mathbb{R}, i=1,2,$$
for any $\underline{w}=(\underline{w}_{1}(x), \underline{w}_{2}(x))$ and $\overline{w}=(\overline{w}_{1}(x), \overline{w}_{2}(x))\in X^{2}.$

We define the set $H\subset X^{2}$ by
$$H=\{(w_{1},w_{2})\in X^{2}: 0\leq w_{1}\leq 1 \text{ and }0\leq w_{2}\leq b-1\}.$$
Our initial datum will always be chosen in the set of $H$. Here, we point out that the set $H$ is positively invariant under the semiflow $\{S(t)\}_{t\geq0}$ generated by system \eqref{1}. In particular, this means that system \eqref{1} with initial condition \eqref{2} admits a unique globally defined solution $(U(x, t), V(x, t))$ with
$$(U, V)(x, \cdot) \in C^{1} ([0,\infty),X^{2}),\quad \forall x\in\R\quad\text{and}\quad (U, V )(\cdot, t) \in H, \quad\forall t \geq 0.$$

Based on the research of Choi et al. \cite{Choi21}, we obtain the spreading speed of the population prey by taking $\alpha\equiv 1$ and $v\equiv0$ in the $u$-equation of \eqref{1}, which is given by the quantity
\begin{equation}\label{5}
s^*:=\inf _{0<\lambda<+\infty} \frac{d_{1}\left[\int_{\mathbb{R}} J_{1}(y) \mathrm{e}^{\lambda y} \mathrm{d}y-1\right]+r_1}{\lambda}.
\end{equation}
Since $b>1$, we obtain the spreading speed of the predator population when the density of prey is fixed to its maximal capacity $1$, namely
\begin{equation}\label{6}
s_{*}:=\inf _{0<\lambda<+\infty} \frac{d_{2}\left[\int_{\mathbb{R}} J_{2}(y) \mathrm{e}^{\lambda y} \mathrm{d} y-1\right]+r_2(b-1)}{\lambda}.
\end{equation}
Let $s>0$ be a given fixed constant, we assume that
$$\underline{s}:=\min\{s^*, s_*\}>s.$$

Next, we give the following proposition for spreading speed of a nonlocal system which will be used in Sections 3 and 4.
\begin{proposition}[\cite{Jin}]\label{spreading speed}
Let w be a solution of the following system
\begin{equation}\label{4}
\left\{\begin{array}{l}
{\frac{\partial w}{\partial t}(x, t)=d(J\ast w-w)+rw(x,t)(k-w(x,t)),\quad x\in\mathbb{R},t>0,} \\
{w(x,0)=\chi(x),\quad x\in\mathbb{R},}
\end{array}\right.
\end{equation}
where kernel $J$ satisfies $(J1)$-$(J2)$ and the initial data $\chi\in X_{s}$ admits a nonempty compact support. The parameters satisfy $d > 0, r > 0$ and $k > 0$. Let
$w(\cdot, t)\in X_{s}$ for all $t > 0$ for a given $\chi\in X_{s}$ and $\bar{c}:=\inf _{0<\lambda<+\infty} \frac{d\left[\int_{\mathbb{R}} J(x) \mathrm{e}^{\lambda x} \mathrm{d} x-1\right]+rk}{\lambda}>0.$ Then the following statements are valid.
\begin{itemize}
\item[\emph{(i)}] For any $c>\bar{c}$, if $\chi$ has a nonempty compact support, then
\begin{align*}
\displaystyle\lim _{t \rightarrow \infty} \sup _{|x|>c t} w(x, t)=0;
\end{align*}
\item [\emph{(ii)}]For any $0<c<\bar{c}$, if $\chi(\cdot)\not\equiv0$, then
\begin{align*}
\displaystyle\liminf _{t \rightarrow \infty} \inf _{|x|<c t} w(x, t)=k.
\end{align*}
\end{itemize}
\end{proposition}
We give the following lemma, which plays an important role in performing various limiting arguments.
\begin{lemma}\label{A-A}
Let $\bar{\alpha}=\max\{-\alpha(-\infty), 1\}$. For any $c\in(s, \underline{s})$, we assume that $(J1), (J2)$, $(\alpha 1)-(\alpha 3), (H1)$ and $d_{1}>r_{1}\bar{\alpha}+\frac{r_1a}{2}+\frac{r_2 b(b-1)}{2}$ and $d_{2}>r_2(b-1)+\frac{r_2 b(b-1)}{2}+\frac{ar_1}{2}$. For any initial data satisfying $(u_0, v_0)\in H$ with $u_0\not\equiv0$ and $v_0\not\equiv0$, the corresponding solution $(u, v)$ of \eqref{1} satisfies
$$(u, v)(x+c t_{n}, t+t_{n})\rightarrow(u_{\infty}, v_{\infty})(x, t)\text{ locally uniformly }\text{ as }n\rightarrow\infty,$$
where $\{t_{n}\}_{n\in\mathbb{Z}}$ be such that $t_{n}\rightarrow\infty$, as $n\rightarrow\infty$ and $(u_{\infty}, v_{\infty})(x, t)$ satisfies
\[
\left\{\begin{array}{l}
\frac{\partial u_{\infty}}{\partial t}(x, t)=d_{1}\left[\left(J_{1} * u_{\infty}\right)(x, t)-u_{\infty}(x, t)\right]+r_{1}u_{\infty}(x, t)\left[1-u_{\infty}(x, t)-av_{\infty}(x, t)\right], \\
\frac{\partial v_{\infty}}{\partial t}(x, t)=d_{2}\left[\left(J_{2} * v_{\infty}\right)(x, t)-v_{\infty}(x, t)\right]+r_2v_{\infty}(x, t)\left[-1+ bu_{\infty}(x, t)-v_{\infty}(x, t)\right].
\end{array}\right.
\]
\end{lemma}
\begin{proof}
Let $\{t_{n}\}_{n\in\mathbb{Z}}$ be such that $t_{n}\rightarrow\infty$, as $n\rightarrow\infty$. Define
\[
\left\{\begin{array}{l}
u_{n}(x, t)=u\left(x+c t_{n}, t+t_{n}\right), \\
v_{n}(x, t)=v\left(x+c t_{n}, t+t_{n}\right),\\
\alpha_n(x-st)=\alpha(x-st+(c-s)t_n),
\end{array}\right.
\]
for $(x, t) \in \mathbb{R} \times\left[-t_{n},+\infty\right).$ It is clearly that $\left(u_{n}(x, t), v_{n}(x, t)\right)$ satisfies
\[
\left\{\begin{array}{l}
\frac{\partial u_{n}}{\partial t}(x, t)=d_{1}\left[\left(J_{1} * u_{n}\right)(x, t)-u_{n}(x, t)\right]+r_{1}u_{n}(x, t)\left[\alpha_n(x-st)-u_{n}(x, t)-av_{n}(x, t)\right], \\
\frac{\partial v_{n}}{\partial t}(x, t)=d_{2}\left[\left(J_{2} * v_{n}\right)(x, t)-v_{n}(x, t)\right]+r_2v_{n}(x, t)\left[-1+ bu_{n}(x, t)-v_n(x, t)\right], \\
u_{n}\left(x,-t_{n}\right)=u\left(x+c t_{n}, 0\right), v_{n}\left(x,-t_{n}\right)=v\left(x+c t_{n}, 0\right).
\end{array}\right.
\]
Next, we use some prior estimates of $(u_{n}(x, t), v_{n}(x, t))$ uniformly in $n$ to reach to the limit as $n\rightarrow +\infty$. Since 
$$
0\leq u(x, 0)=u_0\leq 1, 0\leq v(x, 0)=v_0\leq b-1,
$$
we have 
$$
0\leq u_{n}(x, -t_{n})\leq 1, 0\leq v_{n}(x, -t_{n})\leq b-1.$$ Hence, $0\leq u_{n}(x, t)\leq 1, 0\leq v_{n}(x, t)\leq b-1$. By assumptions $(\alpha_1)$ and $(\alpha_2)$, we have
\begin{equation}\label{m2.4}
	|\alpha_n(x-st)|\leq\max\{-\alpha(-\infty), 1\}=\bar{\alpha}.
\end{equation}
Therefore, there are positive constants $D_{i}, i=1, 2, \ldots, 4,$ such that for $(x, t)\in\mathbb{R}\times[-t_{n}, +\infty)$ and $n\in\mathbb{Z}$,
\[
\begin{array}{l}
\left|\frac{\partial u_{n}}{\partial t}\right| \leq d_{1}\left|J_{1} * u_{n}\right|+d_{1}\left|u_{n}\right|+r_{1}\left|u_{n}\right|\left[|\alpha_n(x-st)|+\left|u_{n}\right|+a\left|v_{n}\right|\right]\\
~~~~~~~\leq 2d_{1}+r_1(1+\bar{\alpha}+a(b-1))=: D_{1}, \\
\left|\frac{\partial v_{n}}{\partial t}\right| \leq d_{2}\left|J_{2} * v_{n}\right|+d_{2}\left|v_{n}\right|+r_2\left|v_{n}\right|\left[1+b\left|u_{n}\right|+\left|v_{n}\right|\right] \leq (b-1)\left(2d_{2}+2 br_2\right)=: D_{2}.
\end{array}
\]

It follows from assumption $(\alpha_3)$ that there exists a constant $M>0$ such that for any $n\in\mathbb{N}$
\begin{equation}\label{2.3m}
|(\alpha_n)_t|\leq M, \text{ for any }t\in\R.
\end{equation}
Therefore, by using \eqref{2.3m} and the assumption $c\in(s, \underline{s})$, we have that
\[
\begin{aligned}
\left|\frac{\partial^{2} u_{n}}{\partial t^{2}}\right|\leq &\left.d_{1} | J_{1} *(u_{n})_{t}|+d_{1}|(u_{n})_{t}|
+r_{1}|(u_{n})_{t}|\left[|\alpha_n(x-st)|+|u_{n}|+a|v_{n}|\right]\right.\\
&+r_{1}\left|u_{n}\right|\left[|-s*(\alpha_n)_t|+\left|\left(u_{n}\right)_{t}\right|)+a\left|(v_{n})_{t}\right|\right] \\
\leq& 2 d_{1} D_{1}+ r_{1}D_{1}(1+\bar{\alpha}+ab-a)+r_{1}(\underline{s}M+D_{1}+aD_{2})=: D_{3}, \\
\left|\frac{\partial^{2} v_{n}}{\partial t^{2}}\right|\leq& 2 d_{2} D_{2}+ 2br_2D_{2}+r_2(b-1)( bD_1+D_{2})=: D_{4}.
\end{aligned}
\]
For any $\gamma>0$, define
\[
\left\{\begin{array}{l}
\mathcal{U}_{n, \gamma}(x, t):=u_{n}(x+\gamma, t)-u_{n}(x, t), \\
\mathcal{V}_{n, \gamma}(x, t):=v_{n}(x+\gamma, t)-v_{n}(x, t), \\
\tilde{J}_{i}(x):=J_{i}(x+\gamma)-J_{i}(x), i=1,2.
\end{array}\right.
\]
Since $J_{i}$ satisfies (J1) and (J2), then $J_{i}^{\prime}\in L^{1}$, there exists $L_{i}>0, i=1,2,$ such that
\[
\begin{aligned}
\int_{\mathbb{R}}\left|\tilde{J}_{i}(x-y)\right| \dy &=\int_{\mathbb{R}}\left|J_{i}(x+\gamma-y)-J_{i}(x-y)\right| \dy \\
&=|\gamma| \int_{\mathbb{R}}\left|\int_{0}^{1} J_{i}^{\prime}(x-y+\theta \gamma) d \theta\right| \dy \\
& \leq|\gamma| \int_{0}^{1} \int_{\mathbb{R}}\left|J_{i}^{\prime}(x-y+\theta \gamma)\right| \dy \d\theta \leq L_{i}|\gamma|.
\end{aligned}
\]
By assumptions ($\alpha_1$) and ($\alpha_3$), there exists $L_3>0$ such that
$$|\alpha_n(x+\gamma-st)-\alpha_n(x-st)|\leq L_3|\gamma|.$$
Hence, for any $\eta> 0$, there exists $\delta_{i} = \frac{\eta}{L_{i}}>0$ $(i=1, 2, 3)$ such that
$$\int_{\mathbb{R}}|\tilde{J}_{i}(x-y)|\dy\leq\eta \text{ and }|\alpha_n(x+\gamma-st)-\alpha_n(x-st)|\leq\eta$$ provided that $|\gamma|\leq \delta_{i}, x\in\mathbb{R}, i = 1, 2, 3.$ Using \eqref{m2.4}, we can verify that
\begin{align}\label{33}
&\frac{\partial}{\partial t} \mathcal{U}_{n, \gamma}^{2}(x, t)=2 \mathcal{U}_{n, \gamma}(x, t) \frac{\partial \mathcal{U}_{n,
\gamma}}{\partial t} (x, t)\nonumber \\
&=2 \mathcal{U}_{n, \gamma}(x, t)\Big(d_{1} \int_{\mathbb{R}} \tilde{J}_{1}(x-y) u_{n}(y,
t)\dy-(d_{1}-r_1\alpha_n(x+\gamma-st))\mathcal{U}_{n, \gamma}(x, t)\nonumber\\
&\quad-r_{1} \mathcal{U}_{n, \gamma}(x, t)(u_{n}(x+\gamma, t)+u_{n}(x, t))+r_1u_n(x, t)(\alpha_n(x+\gamma-st)-\alpha_n(x-st))\nonumber\\
&\quad-r_1a v_{n}(x+\gamma, t)\mathcal{U}_{n, \gamma}(x, t)-r_1au_{n}(x, t) \mathcal{V}_{n, \gamma}(x, t)\Big) \nonumber\\
&\leq 2 \mathcal{U}_{n, \gamma}(x, t)\Big(d_{1} \int_{\mathbb{R}} \tilde{J}_{1}(x-y) u_{n}(y, t)\dy-(d_{1}-r_{1}\alpha_n(x+\gamma-st)) \mathcal{U}_{n, \gamma}(x, t)\nonumber \\
&\quad-r_{1} \mathcal{U}_{n, \gamma}(x, t)(u_{n}(x+\gamma, t)+u_{n}(x, t))+r_1u_n(x, t)(\alpha_n(x+\gamma-st)-\alpha_n(x-st))\Big)\nonumber
\\
&\quad+r_1au_{n}(x, t)\left(\mathcal{U}_{n, \gamma}^{2}(x, t)+\mathcal{V}_{n, \gamma}^{2}(x,
t)\right)\nonumber \\
&\leq 4 (d_{1}+r_1)\eta-2\left(d_{1}-r_{1}\bar{\alpha}-\frac{ar_1}{2}\right) \mathcal{U}_{n, \gamma}^{2}(x, t)+ar_1 \mathcal{V}_{n, \gamma}^{2}(x, t),
\end{align}
and
\begin{align}\label{34}
&\frac{\partial}{\partial t} \mathcal{V}_{n, \gamma}^{2}(x, t)=2 \mathcal{V}_{n, \gamma}(x, t) \frac{\partial \mathcal{V}_{n,
\gamma}}{\partial t} (x, t)\nonumber \\
&=2 \mathcal{V}_{n, \gamma}(x, t)\Big(d_{2} \int_{\mathbb{R}} \tilde{J}_{2}(x-y) v_{n}(y, t)\dy-(d_{2}+r_2) \mathcal{V}_{n, \gamma}(x, t)
\nonumber\\
&\quad-r_2(v_n(x+\gamma, t)+v_n(x, t))\mathcal{V}_{n, \gamma}(x, t)+r_2 bu_{n}(x, t)\mathcal{V}_{n,
\gamma}(x, t)+r_2 bv_{n}(x+\gamma, t)\mathcal{U}_{n, \gamma}(x, t)\Big)\nonumber\\
&\leq2 \mathcal{V}_{n, \gamma}(x, t)\Big(d_{2} \int_{\mathbb{R}} \tilde{J}_{2}(x-y) v_{n}(y, t)\dy-(d_{2}+r_2) \mathcal{V}_{n, \gamma}(x,
t)+r_2 bu_{n}(x, t)\mathcal{V}_{n, \gamma}(x, t)\Big)\nonumber\\
&\quad+r_2 bv_{n}(x+\gamma, t)(\mathcal{U}_{n, \gamma}^{2}(x, t)+\mathcal{V}_{n, \gamma}^{2}(x, t))\nonumber \\
&\leq 4 d_{2} \eta (b-1)  -2\left(d_{2}+r_2-r_2 b-\frac{r_2 b(b-1)}{2}\right) \mathcal{V}_{n, \gamma}^{2}(x, t)+r_2 b(b-1)
\mathcal{U}_{n, \gamma}^{2}(x, t).
\end{align}
Adding the two inequalities \eqref{33} and \eqref{34}, we deduce from the assumptions that $k_1:=d_{1}-r_{1}\bar{\alpha}-\frac{r_1a}{2}-\frac{r_2 b(b-1)}{2}>0$ and $k_2:=d_{2}+r_2-r_2b-\frac{r_2 b(b-1)}{2}-\frac{ar_1}{2}>0$. Then we obtain that
\begin{align}\label{35}
&\frac{\partial}{\partial t}(\mathcal{U}_{n, \gamma}^{2}(x, t)+\mathcal{V}_{n, \gamma}^{2}(x, t))\nonumber \\
& \leq 4\left(d_{1}+r_1+d_{2}(b-1)\right) \eta-2\left(d_{1}-r_{1}\bar{\alpha}-\frac{r_1a}{2}-\frac{r_2 b(b-1)}{2}\right) \mathcal{U}_{n,
\gamma}^{2}(x, t)\nonumber \\
& \quad-2\left(d_{2}+r_2-r_2b-\frac{r_2 b(b-1)}{2}-\frac{ar_1}{2}\right) \mathcal{V}_{n, \gamma}^{2}(x, t)\nonumber \\
&=4\left(d_{1}+r_1+d_{2}(b-1)\right) \eta-2k_{1} \mathcal{U}_{n, \gamma}^{2}(x, t)-2k_{2} \mathcal{V}_{n, \gamma}^{2}(x, t).
\end{align}
Let $k=\min \left\{k_{1}, k_{2}\right\}.$ Due to \eqref{35}, we have
\begin{align}\label{20}
&\frac{\partial}{\partial t}\left(\mathcal{U}_{n, \gamma}^{2}(x, t)+\mathcal{V}_{n, \gamma}^{2}(x, t)\right)\nonumber\\
\leq& 4\left(d_{1}+r_1+d_{2}(b-1)\right) \eta-2k\left(\mathcal{U}_{n, \gamma}^{2}(x, t)+\mathcal{V}_{n, \gamma}^{2}(x, t)\right).
\end{align}
Multiplying both sides of \eqref{20} by $\ee^{2k(t-s)}$ and integrating from $s$ to $t$, we have
\begin{align}\label{37}
&\left(\mathcal{U}_{n, \gamma}^{2}(x, t)+\mathcal{V}_{n, \gamma}^{2}(x, t)\right)\nonumber \\
\leq& \ee^{-2k(t-s)}\left(\mathcal{U}_{n, \gamma}^{2}(x, s)+\mathcal{V}_{n, \gamma}^{2}(x, s)\right)+4\left(d_{1}+r_1+d_{2}(b-1)\right) \eta \int_{s}^{t} \ee^{-2k(t-\theta)} d \theta.
\end{align}
Taking $s=-t_{n}$ from \eqref{37}, we get that
\begin{equation*}
\begin{aligned}
&\left(\mathcal{U}_{n, \gamma}^{2}(x, t)+\mathcal{V}_{n, \gamma}^{2}(x, t)\right) \\
 \leq& \ee^{-2k(t+t_{n})}\left(\mathcal{U}_{n, \gamma}^{2}\left(x,-t_{n}\right)+\mathcal{V}_{n, \gamma}^{2}\left(x,-t_{n}\right)\right)+\frac{2\left(d_{1}+r_1+d_{2}(b-1)\right) \eta}{k}.
\end{aligned}
\end{equation*}
That is,
\begin{equation*}\begin{aligned}
&\left|u_{n}\left(x+\gamma, t\right)-u_{n}\left(x, t\right)\right|^{2}+\left|v_{n}\left(x+\gamma, t\right)-v_{n}\left(x, t\right)\right|^{2} \\
&\leq\left|u_{n}(x+\gamma, -t_{n})-u_{n}(x, -t_{n})\right|^{2}+\left|v_{n}(x+\gamma, -t_{n})-v_{n}(x, -t_{n})\right|^{2}+\frac{2\left(d_{1}+r_1+d_{2}(b-1)\right) \eta}{k}.
\end{aligned}\end{equation*}
Since $u_{n}(x, -t_{n})$ and $v_{n}(x, -t_{n})$ are uniformly continuous for $x \in \mathbb{R}$, there exists $\delta_{4}>0$ such that $\left|u_{n}(x+\gamma, -t_{n})-u_{n}(x, -t_{n})\right| \leq \eta^{1 / 2}$ and $\left|v_{n}(x+\gamma, -t_{n})-v_{n}(x, -t_{n})\right|\leq \eta^{1 / 2}$ whatever $|\gamma| \leq \delta_{4}.$ Thus, there exists a positive constant $D_5$ and for any $\gamma >0$ such that $|\gamma| \leq \delta:=\min \left\{\delta_{1}, \delta_{2}, \delta_{3}, \delta_{4}\right\}$, we have that for all $x\in\mathbb{R}$ and $t>-t_{n}$
\[
\left\{\begin{array}{l}
\left|u_{n}(x+\gamma, t)-u_{n}(x, t)\right|^2 \leq\left(2+\frac{2\left(d_{1}+r_1+d_{2}(b-1)\right)}{k}\right) \eta:=D_{5}\eta, \\
\left|v_{n}(x+\gamma, t)-v_{n}(x, t)\right|^2 \leq\left(2+\frac{2\left(d_{1}+r_1+d_{2}(b-1)\right)}{k}\right) \eta:=D_{5}\eta.
\end{array}\right.
\]
Furthermore, there exist positive constants $D_6$ and $D_7$ such that for all $x\in\mathbb{R}$ and $t>-t_{n}$, it follows that
\begin{equation*}\begin{array}{l}
\left|\frac{\partial u_{n}}{\partial t}(x+\gamma, t)-\frac{\partial u_{n}}{\partial t}(x, t)\right| \\
\begin{aligned}
\leq & \Big| d_{1}\left(J_{1} *\left(u_{n}(x+\gamma, t)-u_{n}(x, t)\right)\right)-d_{1}\left(u_{n}(x+\gamma, t)-u_{n}(x, t)\right) \\
&-r_{1}\left(u_{n}(x+\gamma, t)+u_{n}(x, t)\right)\left(u_{n}(x+\gamma, t)-u_{n}(x, t)\right) \\
&-ar_1v_{n}(x+\gamma, t)\left(u_{n}(x+\gamma, t)-u_{n}(x, t)\right)-ar_1u_{n}(x, t)\left(v_{n}(x+\gamma, t)-v_{n}(x, t)\right) \\
&+\alpha_n(x+\gamma-st)r_1\left(u_{n}(x+\gamma, t)-u_{n}(x, t)\right)+r_1(\alpha_n(x+\gamma-st)-\alpha_n(x-st))u_n(x, t)\Big| \\
\leq &(2 d_{1}+2 r_{1}+ar_1+\bar{\alpha}r_1+ar_1b) D_{5}\eta+r_1\eta=: D_{6} \eta,
\end{aligned}
\end{array}\end{equation*}
and
$$\left|\frac{\partial v_{n}}{\partial t}(x+\gamma, t)-\frac{\partial v_{n}}{\partial t}(x, t)\right| \leq D_{7} \eta.$$
Since $c\in(s, \underline{s})$ and the assumption $(\alpha 2)$, we have that
$$\lim_{n\rightarrow\infty}\alpha_n(x-st)=\lim_{n\rightarrow\infty}\alpha(x-st+(c-s)t_n)=1, $$
locally uniformly with respect to $x\in\R$ and $t\in\R$. By the above prior estimates, the Arzel\`{a}-Ascoli theorem and a diagonal extraction process, we can extract a subsequence $t_{n}\rightarrow\infty$ such that
$$(u_{n}, v_n)(x, t)\rightarrow (u_{\infty}, v_{\infty})(x, t)\text{ locally uniformly }\text{as }n\rightarrow\infty,$$
where $(u_{\infty}, v_{\infty})(x, t)$ satisfies
\[
\left\{\begin{array}{l}
\frac{\partial u_{\infty}}{\partial t}(x, t)=d_{1}\left[\left(J_{1} * u_{\infty}\right)(x, t)-u_{\infty}(x, t)\right]+r_{1}u_{\infty}(x, t)\left[1-u_{\infty}(x, t)-av_{\infty}(x, t)\right], \\
\frac{\partial v_{\infty}}{\partial t}(x, t)=d_{2}\left[\left(J_{2} * v_{\infty}\right)(x, t)-v_{\infty}(x, t)\right]+r_2v_{\infty}(x, t)\left[-1+ bu_{\infty}(x, t)-v_{\infty}(x, t)\right].
\end{array}\right.
\]
We complete the proof of Lemma \ref{A-A}.
\end{proof}

\section{Survival of the prey $u$}
In this section, we consider the large time behavior of solutions of system \eqref{1} with initial value \eqref{2}, and more precisely deal with the persistence of the prey $u$ in the moving frames with speeds between $s$ and
$$
\underline{s}= \min\{s^*, s_*\} > s.
$$
\begin{theorem}[Uniform spreading of $u$]\label{uniform} 
Let $\bar{\alpha}=\max\{-\alpha(-\infty), 1\}$. Assume that $(J1), (J2)$, $(\alpha 1)-(\alpha 3), (H1)$ and $d_{1}>r_{1}\bar{\alpha}+\frac{r_1a}{2}+\frac{r_2 b(b-1)}{2}$ and $d_{2}>r_2(b-1)+\frac{r_2 b(b-1)}{2}+\frac{ar_1}{2}$. Let the initial data $(u_{0}, v_{0})\in H$ with $u_{0}\not\equiv0$ and $v_{0}\not\equiv0$ be given. If $s<\underline{s}$, then for any $\eta\in(0, (\underline{s}-s)/2)$ there exists $\varepsilon > 0$ such that
\begin{equation*}
\displaystyle\liminf _{t \rightarrow+\infty}\inf_{(s+\eta)t\leq|x|\leq(\underline{s}-\eta) t} u(x, t) \geq \varepsilon.
\end{equation*}
\end{theorem}
We divide into three steps to prove Theorem \ref{uniform}. We first use Lemma \ref{pointwise-u} to prove the ``pointwise weak spreading" which illustrates that the $u$-component of the solution of system \eqref{1} does not converge to $0$. Then we apply Lemma \ref{second step} to show the ``pointwise spreading" which means that the solution is bounded along with the path $x=ct$ by some constant $\varepsilon>0$ as $t\rightarrow+\infty.$ We complete the proof of Theorem \ref{uniform} by showing that the spreading is in fact uniform in the intermediate range between the moving frames with speeds $s$ and $\underline{s}=\min\{s_*, s^*\}$.
\begin{lemma}\label{pointwise-u}
(Pointwise weak spreading) Assume that $s<\underline{s}.$ Then for any $c\in(s, \underline{s})$, there exists $\varepsilon_{1}(c)>0$ such that for any $(u_{0}, v_{0})\in H$ with $u_{0}\not\equiv0$ and $v_{0}\not\equiv0$, the corresponding solution $(u, v)$ of \eqref{1} satisfies
$$\limsup_{t\rightarrow\infty}u(ct, t)\geq\varepsilon_{1}(c).$$
\end{lemma}
\begin{proof}
We argue by contradiction by assuming that there exists a sequence
\begin{align*}
 \{(u_{0,n}, v_{0,n})\}_{n\geq0}\in H,
\end{align*}
such that $u_{0,n}, v_{0,n}\not\equiv0$ and
\begin{equation}\label{500}
\lim_{n\rightarrow+\infty}\limsup_{t\rightarrow+\infty}u_{n}(ct, t)=0,
\end{equation}
where $(u_{n}, v_{n})$ is the solution of system \eqref{1} with initial value $(u_{0,n}, v_{0,n})$. 
Then we can choose a time sequence $t_n\rightarrow+\infty$ as $n\rightarrow\infty$ such that
\begin{equation*}
\lim_{n\rightarrow+\infty}\sup_{t\geq t_n}u_{n}(ct, t)=0.
\end{equation*}

Now we claim that for any $R>0$,
\begin{equation}\label{5.4}
\lim_{n\rightarrow+\infty}\sup_{|x|\leq R, t\geq t_n}u_{n}(x+ct, t)=0.
\end{equation}
Indeed, assume by contradiction that there exist sequences $x_n\in[-R, R]$ and $t_n'\geq t_n$ such that
\begin{equation}\label{3.4}
\liminf_{n\rightarrow+\infty}u_{n}(x_n+ct_n', t_n')>0.
\end{equation}
By the arguments similar to the Lemma \ref{A-A} and $c>s$, we can extract a subsequence such that the following convergence holds locally uniform in $(x, t)\in\mathbb{R}\times\mathbb{R}$
\begin{equation*}\left\{\begin{array}{l}
\displaystyle\lim _{n \rightarrow \infty} u_{n}\left(x+ct_{n}', t+t_{n}'\right) = u_{\infty}(x, t),\\
\displaystyle\lim _{n \rightarrow \infty} v_{n}\left(x+ct_{n}', t+t_{n}'\right) = v_{\infty}(x, t),
\end{array}\right.\end{equation*}
where the limit function $(u_{\infty}, v_{\infty})$ is an entire solution of the following system:
\begin{equation}\label{67}
\left\{\begin{array}{l}
{\frac{\partial u_{\infty}}{\partial t}(x, t)=d_{1}\left[\left(J_{1} * u_{\infty}\right)(x, t)-u_{\infty}(x, t)\right]+r_{1}u_{\infty}(x, t)\left[1-u_{\infty}(x, t)-av_{\infty}(x, t)\right],} \\
{\frac{\partial v_{\infty}}{\partial t}(x, t)=d_{2}\left[\left(J_{2} * v_{\infty}\right)(x, t)-v_{\infty}(x, t)\right]+r_2v_{\infty}(x, t)\left[-1+ bu_{\infty}(x, t)-v_{\infty}(x, t)\right].}
\end{array}\right.
\end{equation}
It is easy to see that $u_{\infty}\geq0$ and we deduce from \eqref{500} that $u_{\infty}(0, 0) = 0$. According to the strong maximum principle, we obtain that $u_{\infty}\equiv 0$. On the other hand, by \eqref{3.4}, we can also extract another subsequence $x_n\rightarrow x_{\infty}\in[-R, R]$ such that $u_{\infty}(x_{\infty},0)>0,$ a contradiction. Therefore, \eqref{5.4} holds.

Similarly, we claim that
\begin{equation}\label{5.5}
\lim_{n\rightarrow+\infty}\sup_{|x|\leq R, t\geq t_n}v_{n}(x+ct, t)=0.
\end{equation}
Indeed, if it is not true, then we can find an entire in time solution $(u_{\infty}, v_{\infty})$ of \eqref{67} with $u_{\infty}\equiv0$ and $v_{\infty}>0$.
Then
\begin{equation}\label{a}
\begin{aligned}
\frac{\partial v_{\infty}}{\partial t}(x, t)&=d_{2}\left[\left(J_{2} * v_{\infty}\right)(x, t)-v_{\infty}(x, t)\right]+r_2v_{\infty}(x, t)\left(-1-v_{\infty}(x, t)\right)\\
&\leq d_{2}\left[\left(J_{2} * v_{\infty}\right)(x, t)-v_{\infty}(x, t)\right]-r_2v_{\infty}(x, t),
\end{aligned}
\end{equation}
for all $x\in \R$ and $t\in\R.$ Let $\bar{v}_{\infty}(x, t)$ be the solution of
\begin{equation}\label{c}
\left\{\begin{array}{l}
{\frac{\partial \bar{v}_{\infty}}{\partial t}(x, t)=d_{2}\left[\left(J_{2} * \bar{v}_{\infty}\right)(x, t)-\bar{v}_{\infty}(x, t)\right]-r_2\bar{v}_{\infty}(x, t), x\in\R, t >-t_{0}, t_0\in\R,} \\
{\bar{v}_{\infty}(x, -t_0)=v_{\infty}(x, -t_0), x\in\R, t_0\in\R.}
\end{array}\right.
\end{equation}
It is easy to verify that for any $t_{0}\in\mathbb{R}$, the function $(x, t)\mapsto(b-1)\ee^{-r_2(t+t_{0})}$ is an upper solution of the above equation, for any $t>-t_0.$ By using \eqref{a} and the comparison principle, we have that
\begin{equation}\label{d}
v_{\infty}(x, t)\leq (b-1)\ee^{-r_2(t+t_{0})}, x\in\R, t>-t_0.
\end{equation}
Due to $v_{\infty}(x, -t_{0}) \leq b-1$ for any $t_{0}\in\mathbb{R}^{+}$ and \eqref{d}, we deduce that $v_{\infty}(x, 0)\leq (b-1)\ee^{-r_2t_0}$. Then we have $v_{\infty}(x, 0)\equiv0$ as $t_{0}\rightarrow\infty$. By the strong maximum principle, we obtain that $v_{\infty}\equiv 0$. This contradicts $v_{\infty}>0$. The claim \eqref{5.5} is now proved.

By \eqref{5.5}, for any $\delta_1>0$, there exists $n$ large enough such that
\begin{equation}\label{3.6b}
v_{n}(x, t)\leq\delta_1, \text{ for all }(x, t)\text{ such that }t\geq t_n, x\in(ct-R, ct+R).
\end{equation}
Since $\alpha(\infty)=1,$ and by \eqref{3.6b}, there exists $n$ large enough such that
\begin{equation}\label{3.6c}
\frac{\partial u_{n}}{\partial t}(x, t)\geq d_{1}\left[\left(J_{1} * u_{n}\right)(x, t)-u_{n}(x, t)\right]+r_{1}(1-u_{n}(x, t)-a\delta_1)u_{n}(x, t),
\end{equation}
for all $t\geq t_n$ and $x\in(ct-R, ct+R).$ On the other hand, set
\begin{equation}\label{3.6d}
Q[W](x, t):=d_{1}\int_\R J_1 (x-y) W(y, t)\dy-d_{1}W(x, t)+r_{1}W(x, t)(1-W(x, t)-a\delta_1)-\partial_{t}W(x, t),
\end{equation}
and
\begin{equation}\label{50}
L[W](x, t):=d_{1}\int_\R J_1 (x-y) W(y, t)\dy+m W(x, t)-\partial_{t}W(x, t),
\end{equation}
where $m$ will be determined later.
Let
\begin{equation}\label{52}
\phi(x, t):=\phi_{R,\beta, \eta}(x, t)=
\left\{\begin{array}{l}
\eta\ee^{r_1a\delta_1 t}\ee^{-\beta (x-ct)} \cos \left(\frac{\pi (x-ct)}{2R}\right),\quad \quad x\in(-R+ct, R+ct), t\in\R, \\
0, \quad\quad\quad\quad\quad\quad\quad\quad x\in\R \setminus(-R+ct, R+ct), t\in\R,
\end{array}\right.
\end{equation}
where $\eta\in(0, +\infty)$, $\beta>0$, and we assume that $\beta$ and $\eta$ are two independent constants. Next we show that $\phi(x, t)$ is a sub-solution of \eqref{3.6d}. In fact, we first claim that for $R$ large enough
\begin{align}\label{56}
J_1*\phi(x, t)&=\eta\int_{-R+ct}^{R+ct} J_1(x-y)\ee^{r_1a\delta_1 t} \ee^{-\beta (y-ct)} \cos \left(\frac{\pi (y-ct)}{2R}\right)\dy\nonumber\\
&=\eta\int_{-R}^{R} J_1(x-ct-y)\ee^{r_1a\delta_1 t} \ee^{-\beta y} \cos\left(\frac{\pi y}{2R}\right)\dy\nonumber \\
&\geq \eta\int_{-\infty}^{\infty} J_1(x-ct-y) \ee^{r_1a\delta_1 t}\ee^{-\beta y} \cos\left(\frac{\pi y}{2R}\right)\dy.
\end{align}
Indeed, due to anyhow $J_1*\phi(x, t)\geq0$, we take without loss of generality that $x\in (-R+ct, R+ct)$ and $t\in\R$. In order to make
\eqref{56} valid we have to show that for $y\in\R \setminus(-R+ct, R+ct)$ either $\cos\left(\frac{\pi y}{2R}\right)\leq 0$ or $J_1(x-ct-y)=0$.
We assume that $J_1$ has a compact support, then there exists $K$ such that supp $J_1\subset[-K, K]$. If $x\in(-R+ct, R+ct)$ and $|x-ct-y|\leq K$, then
$y\in (-R-K, R+K)\subset(-3R,3R)$ when $K\leq2R.$ Thus, we obtain that $\cos\left(\frac{\pi y}{2R}\right)\leq 0$ for $y\in[-3R, -R]\cup[R,
3R]$. Moreover, since supp $J_1\subset[-K, K]\subset(-2R, 2R)$ when $K<2R$, we obtain that $J_1(x-ct-y)=0$ for $y\in\R \setminus(-3R, 3R)$.
Thus, we deduce that
\begin{align*}
&\eta\int_{-\infty}^{\infty} J_1(x-ct-y)\ee^{r_1a\delta_1 t} \ee^{-\beta y} \cos\left(\frac{\pi y}{2R}\right)\dy\\
=&\left(\int_{-\infty}^{-3R}+\int_{-3R}^{-R}+\int_{-R}^{R}+\int_{R}^{3R}+\int_{3R}^{\infty}\right) J_1(x-ct-y) \eta\ee^{r_1a\delta_1 t}\ee^{-\beta y} \cos\left(\frac{\pi y}{2R}\right)\dy\\
\leq&\eta\int_{-R}^{R} J_1(x-ct-y)\ee^{r_1a\delta_1 t} \ee^{-\beta y} \cos \left(\frac{\pi y}{2R}\right)\dy,
\end{align*}
that is, \eqref{56} holds. Taking \eqref{52} into \eqref{50} and using \eqref{56}, we obtain that
\begin{align*}
&L[\phi](x, t)\nonumber\\
\geq &c\left[-\beta \phi-\frac{\pi\eta}{2R}\ee^{-\beta (x-ct)}\sin\left(\frac{\pi (x-ct)}{2R}\right)\ee^{r_1a\delta_1 t}\right]+(m-r_1a\delta_1)\phi\\
&+d_{1}\eta\int_{-\infty}^{\infty} J_1(x-ct-y) \ee^{-\beta y} \cos\left(\frac{\pi y}{2R}\right)\ee^{r_1a\delta_1 t}\dy\nonumber\\
=&\left[-c\beta+m-r_1a\delta_1+d_{1}\int_\R \ee^{\beta y} J_1(y) \cos \left(\frac{\pi y}{2R}\right) \dy \right]\phi\nonumber\\
&+\left[-\frac{\pi}{2R}c+d_{1}\int_\R \ee^{\beta y}J_{1}(y)\sin\left(\frac{\pi y}{2R}\right)\dy\right] \eta\ee^{-\beta (x-ct)}\ee^{r_1a\delta_1 t} \sin\left(\frac{\pi (x-ct)}{2R}\right).
\end{align*}
Therefore, $L\left[\phi\right]>0$ on $x\in[-R+ct, R+ct]$ and $t\in\R$ if the following two conditions are satisfied:
\begin{align}\label{59}
c<\frac{1}{\beta}\left[m-r_1a\delta_1+d_{1}\int_\R \ee^{\beta y} J_{1}(y) \cos \left(\frac{\pi y}{2R}\right)\dy\right]=: \mathcal{A}_{m}(\beta, R),
\end{align}
\begin{align}\label{70}
c=\frac{2Rd_{1}}{\pi}\left[\int_\R \ee^{\beta y} J_{1}(y) \sin\left(\frac{\pi y}{2R}\right)\dy\right]=: \mathcal{B}(\beta, R).~~~~~~
\end{align}
We first establish some properties of the functions $\mathcal{A}_{m}$ and $\mathcal{B}$. As $R \rightarrow \infty$, we have locally uniform convergence of
\[
\mathcal{A}_{m}(\beta, R) \rightarrow A_{m}(\beta)=\frac{m-r_1a\delta_1+d_{1}\int_\R \ee^{\beta y} J_{1}(y)\dy}{\beta}, \quad \mathcal{B}(\beta, R) \rightarrow B(\beta):=d_1\int_\R y \ee^{\beta y} J_{1}(y) \dy.
\]
Differentiation gives
\begin{align}\label{89}
&A_{m}^{\prime}(\beta)=\left(B(\beta)-A_{m}(\beta)\right) / \beta,\\
&B^{\prime}(\beta) =d_1\int_\R J(y) \ee^{\beta y} y^2\dy >0\nonumber.
\end{align}
It follows from the properties of the function $A_{m}(\beta)$ that it achieves infimum. Then, there exists $\beta^*>0$ such that $A_{m}(\beta^*)=\inf_{\beta>0}A_{m}(\beta)$. By the definition of $\beta^*$ and \eqref{89}, we obtain that
$B(\beta^*)=A_{m}(\beta^*).$ Since $B$ is an increasing function, thus $B(\beta)<B(\beta^*)$ for $0<\beta<\beta^*.$ Then we have
\begin{align}\label{15}
A_{m}(\beta)>A_{m}(\beta^*)=B(\beta^*)>B(\beta) \text{ for } 0<\beta<\beta^*.
\end{align}
In addition, we define $c^*:=A_{m^*}(\beta^*)$ with $m^*=r_1-r_1\delta_2-r_1a\delta_1-d_1$ and $r_1-r_1\delta_2-2r_1a\delta_1>0$ for small enough constants $\delta_1>0$, $\delta_2>0$. Since $0<s<c<\underline{s},$ we can choose $m<m^*$ such that $0<s< c<A_{m}(\beta^*)<A_{m^*}(\beta^*)=c^*$. Note that $B(0)<c^*$ and $B(0)=0$, we have that $c>s>B(0)=0$.
Then, combined with \eqref{15}, we can choose $c_1, c_2, \delta, R>0$ such that
$$B(c_1)+\delta<c<B(c_2)-\delta\text{ and }|\mathcal{B}(\beta, R)-B(\beta)|<\delta.$$
It follows from the continuity of $\mathcal{B}(\beta, R)$ and $B(\beta)$ that there exists some $\beta(R)$ such that $\mathcal{B}(\beta(R), R)=c$ for all large enough $R$. Obviously, we can choose $R$ large enough such that $\mathcal{A}_m(\beta(R), R) > c$. Thus, we prove that \eqref{70} and \eqref{59} hold true.

By the definition of $\phi_{R,\beta,\eta}(x, t)$, we obtain $L[\phi](x, t)>0$ for $(x, t)\in\R\times\R$. Note that $r_{1}W(1-W-a\delta_1)\geq (r_1-r_1\delta_2-r_1a\delta_1)W$ for $0\leq W\leq\delta_2$ and $m<r_1-r_1\delta_2-r_1a\delta_1-d_1$. Therefore, we have
$$
Q\left[\phi\right](x, t)>L\left[\phi\right](x, t)>0, \text{ for }(x, t)\in\R\times\R,
$$
namely,
\begin{equation*}
-d_{1}\int_\R J_1 (x-y) \phi(y, t)\dy-r_{1}\phi(x, t)(1-\phi(x, t)-a\delta_1)+\partial_{t}\phi(x, t)<0,\text{ for } (x, t)\in\R\times\R.
\end{equation*}
Using \eqref{3.6c} and taking $\eta$ small enough so that
$u_n(x,t_n)\geq \phi(x,t_n), \forall x\in\R$, we get by the comparison principle that
$$u_n(x, t)\geq \phi(x, t)$$
for all $t\geq t_n$ and $x\in\R.$ However, $\phi(ct, t)=\eta\ee^{r_1a\delta_1 t}\rightarrow+\infty$ as $t\rightarrow+\infty,$ which contradicts $0\leq u_n\leq 1$. We complete the proof of Lemma \ref{pointwise-u}.
\end{proof}
\begin{lemma}\label{c=0}
For any $c\in[0, \underline{s})$, there exists $\varepsilon_{1}'(c) >0$ such that, for any initial data satisfying $(u_0, v_0)\in H$ with $u_{0}\not\equiv0$ and $v_{0}\not\equiv0$, the solution $(u, v)$ of the following system
\begin{equation}\label{3.13}
\left\{\begin{array}{l}
{\frac{\partial u}{\partial t}(x, t)=d_{1}\left[\left(J_{1} * u\right)(x, t)-u(x, t)\right]+r_{1}u(x, t)\left[1-u(x, t)-av(x, t)\right],} \\
{\frac{\partial v}{\partial t}(x, t)=d_{2}\left[\left(J_{2} * v\right)(x, t)-v(x, t)\right]+r_2v(x, t)\left[-1+ bu(x, t)-v(x, t)\right].}
\end{array}\right.
\end{equation}
satisfies
$$
\limsup_{t\rightarrow+\infty}u(ct,t)\geq \varepsilon_{1}'(c).
$$
\end{lemma}
\begin{proof}
  The proof of Lemma \ref{c=0} is similar to Lemma \ref{pointwise-u}, so we omit it.
\end{proof}

\begin{lemma}\label{second step}
(Pointwise spreading) Assume that $s<\underline{s}.$ Then for any $(u_{0}, v_{0})\in H$ with $u_{0}\not\equiv0$ and $v_{0}\not\equiv0$, for all $c\in(s, \underline{s})$, there exists $\varepsilon_{2}(c) > 0$ such that the corresponding solution $(u, v)$ of \eqref{1} satisfies
\begin{equation*}
\displaystyle\liminf _{t \rightarrow+\infty} u(ct, t) \geq \varepsilon_{2}\left(c\right).
\end{equation*}
\end{lemma}
\begin{proof}
We argue by contradiction to prove this assertion, i.e. $u$ spreads away from $0$. We assume that there are  sequences $(u_{0,n}, v_{0,n})\in H$ with $u_{0, n}\not\equiv0$ and $v_{0, n}\not\equiv0$, such that
\begin{equation}\label{3.17b}
\lim_{n \rightarrow+\infty} u_{n}(ct_n, t_n)=0.
\end{equation}
By Lemma \ref{pointwise-u}, there exists another sequence $t_{n}'\rightarrow\infty$ such that
$$\lim_{n \rightarrow+\infty} u_{n}(ct_n', t_n')\geq \frac{\varepsilon_1(c)}{2},$$
and without loss of generality we can choose it so that $t_n'<t_n$ for any $n$. We define
$$\tau_n:=\sup\left\{t_n^{\prime} \leq t \leq t_n|u_{n}(ct, t)\geq \frac{\varepsilon_1(c)}{2}\right\},$$
which follows that
\begin{equation}\label{3.17d}
\forall t\in(\tau_n, t_n), u_{n}(ct, t)\leq \frac{\varepsilon_1(c)}{2}.
\end{equation}
Then this yields the following properties:
\begin{equation}\label{k}
\begin{aligned}
&u_{n}(c\tau_n, \tau_n)= \frac{\varepsilon_1(c)}{2},\\
&u_{n}(ct, t)\leq \frac{\varepsilon_1(c)}{2}, \quad t\in(\tau_n, t_n),\\
&u_{n}(c(\tau_n+t_n-\tau_n), \tau_n+t_n-\tau_n)\rightarrow 0, \text{ as }n\rightarrow \infty.
\end{aligned}
\end{equation}
By Lemma \ref{A-A}, we can extract a subsequence $t_{n}\rightarrow\infty$ such that
$$(u_{n}, v_n)(x+ct_n, t+t_n)\rightarrow (u_{\infty}, v_{\infty})(x, t)\text{ locally uniformly }\text{as }n\rightarrow\infty,$$
where $(u_{\infty}, v_{\infty})(x, t)$ satisfies
\[
\left\{\begin{array}{l}
\frac{\partial u_{\infty}}{\partial t}(x, t)=d_{1}\left[\left(J_{1} * u_{\infty}\right)(x, t)-u_{\infty}(x, t)\right]+r_{1}u_{\infty}(x, t)\left[1-u_{\infty}(x, t)-av_{\infty}(x, t)\right], \\
\frac{\partial v_{\infty}}{\partial t}(x, t)=d_{2}\left[\left(J_{2} * v_{\infty}\right)(x, t)-v_{\infty}(x, t)\right]+r_2v_{\infty}(x, t)\left[-1+ bu_{\infty}(x, t)-v_{\infty}(x, t)\right].
\end{array}\right.
\]
From the choice of $t_n$, we have $u_{\infty}(0,0)=0,$ hence $u_{\infty}\equiv0$ by the strong maximum principle. In particular, the sequence $t_n-\tau_n$ is unbounded. Indeed, assume by contraction that $\lim_{n\rightarrow\infty} (t_n-\tau_n)=l<+\infty$, then it follows from the third formula of \eqref{k} that
$$\lim_{n\rightarrow+\infty}u_{n}(c(\tau_n+t), \tau_n+t)=0,\quad \forall t\in[0, t_n-\tau_n],$$
which contradicts the fact that
$$u_{n}(c\tau_n, \tau_n)=\frac{\varepsilon_1(c)}{2},$$
for $n$ large enough. Thus we obtain that $t_n-\tau_n\rightarrow+\infty,$ as $n\rightarrow+\infty.$
Therefore we can extract a subsequence such that
\begin{equation*}\left\{\begin{array}{l}
\displaystyle \tilde{u}(x, t)=\lim_{n \rightarrow+\infty} u_{n}(x+c\tau_{n}, \tau_{n}+t), \\
\displaystyle \tilde{v}(x, t)=\lim_{n \rightarrow+\infty} v_{n}(x+c\tau_{n}, \tau_{n}+t),
\end{array}\right.\end{equation*}
which are well defined as a result of global boundedness and a prior estimate. The pair $(\tilde{u}, \tilde{v})$ is a global in time solution of system \eqref{3.13}, and $\tilde{u} (0, 0) = \frac{\varepsilon_1(c)}{2}>0$.
In addition, we have
\begin{equation}\label{5.6}
\tilde{u} (ct, t) \leq\frac{\varepsilon_1(c)}{2}, \text{ for all }t\geq0.
\end{equation}
Notice that, when $\tilde{v}\not\equiv0,$ this entire solution immediately contradicts Lemma \ref{c=0}. Now it remains to consider the case when $\tilde{v}\equiv0.$ Then
$$\frac{\partial \tilde{u}}{\partial t}(x, t)=d_{1}\left[\left(J_{1} * \tilde{u}\right)(x, t)-\tilde{u}(x, t)\right]+r_{1}\tilde{u}(x, t)(1-\tilde{u}(x, t)),$$
which is the scalar nonlocal diffusion equation of KPP type. Recall that $c<\underline{s}=\min\{s_*, s^*\}\leq s^*,$ and $\tilde{u}(x, 0)\geq\not\equiv 0$,
then by Proposition \ref{spreading speed}, we have that
$$\tilde{u} (ct, t)\rightarrow1, \text{ as }t\rightarrow+\infty,$$
which contradicts \eqref{5.6}. Thus we complete the proof of Lemma \ref{second step}.
\end{proof}

\begin{proof}[Proof of Theorem \ref{uniform}]
We fix $\eta$ and argue by contradiction by assuming that there exist $\{t_{n, k}\}$ and  $\{x_{n, k}\}$ with $t_{n, k}\rightarrow+\infty$, as $k\rightarrow+\infty,$ and
$$x_{n, k}\in[(s+\eta)t_{n, k}, (\underline{s}-\eta) t_{n, k}],$$
such that
\begin{equation}\label{701}
u_n(x_{n, k}, t_{n,k})\leq\frac{1}{n},
\end{equation}
for any positive integers $n$ and $k$. However, applying Lemma \ref{second step}, one has that
\begin{equation*}
\liminf_{t \rightarrow+\infty}u_n((\underline{s}-\eta/2)t, t)\geq\varepsilon_2(\underline{s}-\eta/2).
\end{equation*}
In particular, we define another time sequence
$$t_{n,k}':= \frac{x_{n,k}}{\underline{s}-\eta/2}\in[0, t_{n,k}),\quad \forall n\geq0,\quad \lim_{k\rightarrow+\infty}t_{n,k}'=+\infty.$$
Then applying Lemma \ref{second step}, one has that
\begin{equation*}
u_n(x_{n,k}, t_{n,k}')\geq\frac{\varepsilon_2(\underline{s}-\eta/2)}{2},
\end{equation*}
for any $k$ large enough. For each $n$, we choose such a large $k$ and drop it from our notation for convenience. Then we can define
\[
\begin{array}{l}
\tau_{n}:=\sup \left\{t_n^{\prime} \leq t \leq t_n| u_{n}\left(ct, t\right) \geq \frac{\min\{\varepsilon_{1}'(c), \varepsilon_2(\underline{s}-\eta/2)\}}{2}\right\},
\end{array}
\]
where $\varepsilon_{1}'(c)$ comes from Lemma \ref{c=0}. Similar to the proof of Lemma \ref{second step}, we have that
$$t_n-\tau_n\rightarrow+\infty, \text{ as }n\rightarrow+\infty.$$
We consider the functions
\begin{equation*}
\displaystyle \tilde{u}_{n}(x, t)=u(x+c \tau_{n}, \tau_n+t),\quad\tilde{v}_{n}(x, t)=v(x+c \tau_{n}, \tau_n+t),
\end{equation*}
and
$$\displaystyle \lim_{n\rightarrow\infty}\tilde{u}_{n}(x, t):=\tilde{u}(x, t), \quad\lim_{n\rightarrow\infty}\tilde{v}_{n}(x, t):=\tilde{v}(x, t).$$
Then the following properties can be derived,
\[
\begin{array}{l}
\tilde{u}\left(0, 0\right) = \frac{\min\{\varepsilon_{1}'(c), \varepsilon_2(\underline{s}-\eta/2)\}}{2},\quad
\tilde{u}\left(0, t\right) \leq \frac{\min\{\varepsilon_{1}'(c), \varepsilon_2(\underline{s}-\eta/2)\}}{2}, \text{ for all }t\geq0.
\end{array}
\]
Regardless of whether $\tilde{v}\equiv0$ or $\tilde{v} > 0$, we reach a contradiction with either the result of
Proposition \ref{spreading speed}, or Lemma \ref{c=0}. This concludes the proof.

\end{proof}

\section{Survival of the predator $v$}
In this section, we show the persistence of the predator $v$ in the moving frames with speeds in the interval $(s, \underline{s})$, where we recall that $\underline{s}=\min\{s_*, s^*\}$.
\begin{theorem}[Uniform spreading of $v$]\label{uniform-v}
Let $\bar{\alpha}=\max\{-\alpha(-\infty), 1\}$. Assume that $(J1), (J2)$, $(\alpha 1)-(\alpha 3), (H1)$ and $d_{1}>r_{1}\bar{\alpha}+\frac{r_1a}{2}+\frac{r_2 b(b-1)}{2}$ and $d_{2}>r_2(b-1)+\frac{r_2 b(b-1)}{2}+\frac{ar_1}{2}$. Let initial data $(u_{0}, v_{0})\in H$ with $u_{0}\not\equiv0$ and $v_{0}\not\equiv0$ be given. If $s<\underline{s}$, then for any $\eta\in(0, (\underline{s}-s)/2)$ there exists $\varepsilon > 0$ such that
\begin{equation*}
\displaystyle\liminf _{t \rightarrow+\infty}\inf_{(s+\eta)t\leq|x|\leq(\underline{s}-\eta) t} v(x, t) \geq \varepsilon.
\end{equation*}
\end{theorem}
The method is the same as for the prey, before proving Theorem \ref{uniform-v}, we give the following three lemmas.
\begin{lemma}\label{pointwise}
(Pointwise weak spreading) Assume that $s<\underline{s}.$ Then for any $c\in(s, \underline{s})$, there exists $\varepsilon_{3}(c)>0$ such that for any $(u_{0}, v_{0})\in H$ with $u_{0}\not\equiv0$ and $v_{0}\not\equiv0$, the corresponding solution $(u, v)$ of \eqref{1} satisfies
$$\limsup_{t\rightarrow\infty}v(ct, t)\geq\varepsilon_{3}(c).$$
\end{lemma}
\begin{proof}
We let $c\in(s, \underline{s}),$ and assume by contradiction that there exists a sequence of solutions $\{(\tilde{u}_n, \tilde{v}_n)\}$ with
initial data $\{u_{0,n}, v_{0,n}\}\subset H$, with $u_{0,n}\not\equiv0$ and $v_{0,n}\not\equiv0$ such that
\begin{equation}\label{4.1f}
\lim_{n\rightarrow+\infty}\limsup_{t\rightarrow\infty}\tilde{v}_n(ct, t)=0.
\end{equation}
Then for each $n$, we can choose $t_n$ large enough such that
\begin{equation}\label{5.8}
\lim_{n\rightarrow+\infty}\sup_{t\geq t_n}\tilde{v}_n(ct, t)=0.
\end{equation}

Next we claim that for any $R>0$,
\begin{equation}\label{5.9}
\limsup_{n\rightarrow+\infty}\{\sup_{t\geq t_n, |x-ct|\leq R}\tilde{v}_n(x, t)\}=0.
\end{equation}
Indeed, assume by contradiction that there exist sequences $t_n'\geq t_n$ and $x_n\in[ct_n'-R, ct_n'+R]$ such that
\begin{equation}\label{3.4k}
\liminf_{n\rightarrow+\infty}\tilde{v}_{n}(x_n, t_n')>0.
\end{equation}
By using the arguments similar to Lemma \ref{A-A}, we can extract a subsequence such that the following convergence holds locally uniformly in $(x, t)\in\mathbb{R}\times\mathbb{R}$
\begin{equation*}\left\{\begin{array}{l}
\displaystyle\lim _{n \rightarrow \infty} \tilde{u}_{n}\left(x+x_n, t+t_{n}'\right) = \tilde{u}_{\infty}(x, t),\\
\displaystyle\lim _{n \rightarrow \infty} \tilde{v}_{n}\left(x+x_n, t+t_{n}'\right) = \tilde{v}_{\infty}(x, t),
\end{array}\right.\end{equation*}
where the limit function $(\tilde{u}_{\infty}, \tilde{v}_{\infty})$ is an entire solution of the following system:
\begin{equation}\label{670}
\left\{\begin{array}{l}
{\frac{\partial \tilde{u}_{\infty}}{\partial t}(x, t)=d_{1}\left[\left(J_{1} * \tilde{u}_{\infty}\right)(x, t)-\tilde{u}_{\infty}(x, t)\right]+r_{1}\tilde{u}_{\infty}(x, t)\left[1-\tilde{u}_{\infty}(x, t)-a\tilde{v}_{\infty}(x, t)\right],} \\
{\frac{\partial \tilde{v}_{\infty}}{\partial t}(x, t)=d_{2}\left[\left(J_{2} * \tilde{v}_{\infty}\right)(x, t)-\tilde{v}_{\infty}(x, t)\right]+r_2\tilde{v}_{\infty}(x, t)\left[-1+ b\tilde{u}_{\infty}(x, t)-\tilde{v}_{\infty}(x, t)\right].}
\end{array}\right.
\end{equation}
It is clearly that $\tilde{v}_{\infty}\geq0$. By extracting subsequences $t_n'\geq t_n$, $x_n\rightarrow x_{\infty}\in[ct_{n}'-R, ct_{n}'+R]$ and using \eqref{4.1f}, we deduce that $\tilde{v}_{\infty}(ct_{n}'-x_{\infty}, 0) = 0$. According to the strong maximum principle, we obtain that $\tilde{v}_{\infty}\equiv 0$. On the other hand, by \eqref{3.4k}, we deduce that $\tilde{v}_{\infty}(0,0)>0,$ a contradiction. Therefore, \eqref{5.9} holds.

Now, we claim that
\begin{equation}\label{5.10}
\limsup_{n\rightarrow+\infty}\{\sup_{t\geq t_n, |x-ct|\leq R}\tilde{u}_n(x, t)\}=1, \text{ for any }R>0.
\end{equation}
We assume by contradiction that there is a sequence $\{(x_n, t_n')\}$ with $t_n'\geq t_n$ and $x_n\in[ct_n'-R, ct_n'+R]$ such that
$$
\limsup_{n\rightarrow+\infty}\tilde{u}_n(x_n, t_n')<1.
$$
By the prior estimates, the Arzel\`{a}-Ascoli theorem and a diagonal extraction process, we can extract a subsequence $t_{n}'\rightarrow\infty$ such that
$$(\tilde{u}_{n}, \tilde{v}_n)(x+x_n, t+t_n')\rightarrow (u_{\infty}, v_{\infty})(x, t)\text{ locally uniformly }\text{as }n\rightarrow\infty,$$
where $(u_{\infty}, v_{\infty})(x, t)$ satisfies
\[
\left\{\begin{array}{l}
\frac{\partial u_{\infty}}{\partial t}(x, t)=d_{1}\left[\left(J_{1} * u_{\infty}\right)(x, t)-u_{\infty}(x, t)\right]+r_{1}u_{\infty}(x, t)\left[1-u_{\infty}(x, t)-av_{\infty}(x, t)\right], \\
\frac{\partial v_{\infty}}{\partial t}(x, t)=d_{2}\left[\left(J_{2} * v_{\infty}\right)(x, t)-v_{\infty}(x, t)\right]+r_2v_{\infty}(x, t)\left[-1+ bu_{\infty}(x, t)-v_{\infty}(x, t)\right].
\end{array}\right.
\]
Since $v_{\infty}(0, t)=0$ for all $t>0,$ by the strong maximum principle we get that $v_{\infty}\equiv 0$. In particular, $u_{\infty}$ satisfies
$$\frac{\partial u_{\infty}}{\partial t}(x, t)=d_{1}\left[\left(J_{1} * u_{\infty}\right)(x, t)-u_{\infty}(x, t)\right]+r_{1}u_{\infty}(x, t)\left(1-u_{\infty}(x, t)\right), \quad (x, t)\in\R^2.$$
On the other hand, by Theorem \ref{uniform}, we have that
$$\inf_{(x, t)\in\R^2} u_{\infty}(x, t)>0.$$
This implies that $u_{\infty}\equiv1,$ a contradiction to $u_{\infty}(0,0)<1$ by our choices of $x_n$ and $t_n'$. Hence \eqref{5.10} holds.

For any small $\delta_1>0$ and large $R>0$, we have
$$\frac{\partial \tilde{v}_{n}}{\partial t}(x, t)\geq d_{2}\left[\left(J_{2} * \tilde{v}_{n}\right)(x, t)-\tilde{v}_{n}(x, t)\right]+r_2\tilde{v}_{n}(x, t)\left(-1+ b-\delta_1\right), |x-ct_n|\leq R, t\geq t_n,$$
for any $n$ large enough. Similar to the proof of Lemma \ref{pointwise-u}, we infer that $\tilde{v}_{n}(ct, t)\rightarrow+\infty$ as $t\rightarrow+\infty$, a contradiction to \eqref{5.9}. We complete the proof of Lemma \ref{pointwise}.
\end{proof}
\begin{lemma}\label{c=0-v}
For any $c\in[0, \underline{s})$, there exists $\varepsilon_{3}'(c) >0$ such that, for any initial data satisfying $(u_0, v_0)\in H$ with $u_{0}\not\equiv0$ and $v_{0}\not\equiv0$, the corresponding solution $(u, v)$ of \eqref{3.13}
satisfies
$$
\limsup_{t\rightarrow+\infty}u(ct,t)\geq \varepsilon_{3}'(c).
$$
\end{lemma}
\begin{proof}
  The proof of Lemma \ref{c=0-v} is similar to Lemma \ref{pointwise}, so we omit it.
\end{proof}
\begin{lemma}\label{second step-v}
(Pointwise spreading) Assume that $s<\underline{s}.$ Then for any $(u_{0}, v_{0})\in H$ with $u_{0}\not\equiv0$ and $v_{0}\not\equiv0$, for all $c\in(s, \underline{s})$, there exists $\varepsilon_{4}(c) > 0$ such that the corresponding solution $(u, v)$ of \eqref{1} satisfies
\begin{equation*}
\displaystyle\liminf _{t \rightarrow+\infty} v(ct, t) \geq \varepsilon_{4}\left(c\right).
\end{equation*}
\end{lemma}
\begin{proof}
The method is the same as that of Lemma \ref{second step}. Fixing $c\in(s, \underline{s})$ and proceeding by contradiction, we find an entire solution $(u_{\infty}, v_{\infty})$ of \eqref{3.13}, such that $v_{\infty} (0, 0) = \frac{\varepsilon_3(c)}{2}>0$ and
\begin{align}\label{5.11}
v_{\infty}(ct, t) \leq\frac{\varepsilon_3(c)}{2}, \forall t\geq0.
\end{align}
Moreover, by Theorem \ref{uniform}, we have that
$$\inf_{(x, t)\in\R^2}u_{\infty}(x, t)>0,$$
which implies that $u_{\infty}\equiv1.$ Thus
$$\frac{\partial v_{\infty}}{\partial t}(x, t)=d_{2}\left[\left(J_{2} * v_{\infty}\right)(x, t)-v_{\infty}(x, t)\right]+r_2v_{\infty}(x, t)\left[-1+ b-v_{\infty}(x, t)\right].$$
Since $c<\underline{s}\leq s^*$, and by Proposition \ref{c=0} and Assumption (H1), we have that
\begin{equation}\label{m}
\lim_{t\rightarrow+\infty} v_{\infty}(ct, t)=b-1>0.
\end{equation}
which contradicts \eqref{5.11}. Thus we complete the proof of Lemma \ref{second step-v}.
\end{proof}

\begin{proof}[Proof of Theorem \ref{uniform-v}]
We fix $\eta$ and argue by contradiction by assuming that there exist $\{t_{n, k}\}$ and  $\{x_{n, k}\}$ with $t_{n, k}\rightarrow+\infty$, as $k\rightarrow+\infty,$ and
$$x_{n, k}\in[(s+\eta)t_{n, k}, (\underline{s}-\eta) t_{n, k}],$$
such that
\begin{equation*}
v_n(x_{n, k}, t_{n,k})\leq\frac{1}{n},
\end{equation*}
for any positive integers $n$ and $k$. However, applying Lemma \ref{second step-v}, we have that
\begin{equation*}
\liminf_{t \rightarrow+\infty}v_n((\underline{s}-\eta/2)t, t)\geq\varepsilon_4(\underline{s}-\eta/2).
\end{equation*}
In particular, we define another time sequence
$$t_{n,k}':= \frac{x_{n,k}}{\underline{s}-\eta/2}\in[0, t_{n,k}),\quad \forall n\geq0,\quad \lim_{k\rightarrow+\infty}t_{n,k}'=+\infty.$$
Then applying Lemma \ref{second step-v}, one has that
\begin{equation*}
v_n(x_{n,k}, t_{n,k}')\geq\frac{\varepsilon_4(\underline{s}-\eta/2)}{2},
\end{equation*}
for any $k$ large enough. For each $n$, we choose such a large $k$ and drop it from our notation for convenience. Then we can define
\[
\begin{array}{l}
\tau_{n}:=\sup \left\{t_n^{\prime} \leq t \leq t_n| v_{n}\left(ct, t\right) \geq \frac{\min\{\varepsilon_{3}'(c), \varepsilon_4(\underline{s}-\eta/2)\}}{2}\right\},
\end{array}
\]
where $\varepsilon_{3}'(c)$ comes from Lemma \ref{c=0-v}. Similar to the proof of Lemma \ref{second step}, we have that
$$t_n-\tau_n\rightarrow+\infty, \text{ as }n\rightarrow+\infty.$$
We consider the functions
\begin{equation}\label{3.32}
\displaystyle \tilde{u}_{n}(x, t)=u(x+c \tau_{n}, \tau_n+t),\quad\tilde{v}_{n}(x, t)=v(x+c \tau_{n}, \tau_n+t),
\end{equation}
and
$$\displaystyle \lim_{n\rightarrow\infty}\tilde{u}_{n}(x, t):=\tilde{u}(x, t), \quad\lim_{n\rightarrow\infty}\tilde{v}_{n}(x, t):=\tilde{v}(x, t).$$
Then we deduce that
\[
\begin{array}{l}
\tilde{v}\left(0, 0\right) = \frac{\min\{\varepsilon_{3}'(c), \varepsilon_4(\underline{s}-\eta/2)\}}{2},\quad
\tilde{v}\left(0, t\right) \leq \frac{\min\{\varepsilon_{3}'(c), \varepsilon_4(\underline{s}-\eta/2)\}}{2}, \text{ for all }t\geq0.
\end{array}
\]
Regardless of whether $\tilde{v}\equiv0$ or $\tilde{v} > 0$, we reach a contradiction with either the result of
Proposition \ref{spreading speed}, or Lemma \ref{c=0-v}. We complete the proof of Theorem \ref{uniform-v}.
\end{proof}

\section*{Acknowledgments}
We are very grateful to the anonymous referee for careful reading and helpful suggestions which led to an improvement of our original manuscript. This work is supported by the National Natural Science Foundation of China (No. 12171039 and 12271044).

\section*{Data availability statement}
The authors declare that no data was used in this paper.

\section*{Conflict of interest statement}
The authors declare that they have no known competing financial interests or personal relationships that could have appeared to influence the work reported in this paper.

\end{document}